\documentclass[10pt]{article}
 \usepackage{amsfonts, epsfig, amsmath, amssymb, color,amsthm,mathabx}
\usepackage{mathrsfs}
\usepackage[english]{babel}

\usepackage[round]{natbib}   

\newcommand{\mbR}{\mathbb{R}}

\newcommand{\E}{\mathbf{E}}
\renewcommand{\P}{\mathbf{P}}

\usepackage{color}

\renewcommand {\epsilon}{\varepsilon}

\theoremstyle{plain}
\newtheorem{thm}{Theorem}[section]
\newtheorem{lem}[thm]{Lemma}

\newtheorem{cor}[thm]{Corollary}

\theoremstyle{definition}
\newtheorem{rem}[thm]{Remark}

\DeclareMathSymbol{\ophi}{\mathalpha}{letters}{"1E}

\newcommand{\e}{\varepsilon}
\newcommand{\ve}{\varepsilon}

\renewcommand{\phi}{\varphi}

\newcommand{\be}{\begin{equation}}
\newcommand{\ee}{\end{equation}}
\newcommand{\ben}{\begin{equation*}}
\newcommand{\een}{\end{equation*}}

\newcommand{\ba}{\begin{equation}\begin{aligned}}
\newcommand{\ea}{\end{aligned}\end{equation}}

\DeclareMathOperator{\Law}{Law}

\renewcommand{\i}{\mathrm{i}}
\newcommand{\ex}{\mathrm{e}}
\newcommand{\di}{\mathrm{d}}

\newcommand{\cF}{\mathcal{F}}

\newcommand{\bI}{\mathbb{I}}

\newcommand{\bR}{\mathbb{R}}

\newfont{\cyrfnt}{wncyr10}
\def\J3{\cyrfnt{\rm \u{\cyrfnt I}}}
\def\j3{\cyrfnt{\rm \u{\cyrfnt i}}}

\allowdisplaybreaks[4]

\let\oldmarginpar\marginpar
\renewcommand{\marginpar}[1]{\oldmarginpar{\scriptsize\texttt{\color{red}{#1}}}}

\numberwithin{equation}{section}

\begin{document}
\title{Generalized selection problem with L\'evy noise}

\date{\today}


\author{Ilya Pavlyukevich\footnote{Institute for Mathematics, Friedrich Schiller University Jena, Ernst--Abbe--Platz 2,
07743 Jena, Germany; ilya.pavlyukevich@uni-jena.de} \ 
 and 
Andrey Pilipenko\footnote{Institute of Mathematics, National Academy of Sciences of Ukraine, Tereshchenkivska Str.\ 3, 01601, Kiev, 
Ukraine} \footnote{National Technical University of Ukraine 
``Igor Sikorsky Kyiv Polytechnic Institute'',
ave.\ Pobedy 37, Kiev 03056, Ukraine; pilipenko.ay@gmail.com}}


\maketitle

\begin{abstract}
Let $A_\pm>0$, $\beta\in(0,1)$, and let $Z^{(\alpha)}$ be a strictly $\alpha$-stable L\'evy process with
the jump measure $\nu(\di z)=(C_+\bI_{(0,\infty)}(z)+ C_-\bI_{(-\infty,0)}(z))|z|^{-1-\alpha}\,\di z$, $\alpha\in (1,2)$, $C_\pm\geq 0$, $C_++C_->0$.
The selection problem for the \emph{model} stochastic differential equation $\di \bar X^\e=(A_+\bI_{[0,\infty)}(\bar X^\e)- A_-\bI_{(-\infty,0)}(\bar X^\e))|\bar X^\e|^\beta \,\di t+\e \di Z^{(\alpha)}$ states that in the small noise limit $\e\to 0$, solutions $\bar X^\e$ converge weakly 
to the maximal or minimal solutions of the limiting non-Lipschitzian ordinary 
differential equation $\di \bar x=(A_+\bI_{[0,\infty)}(\bar x)- A_-\bI_{(\infty,0)}(\bar x))|\bar x|^\beta \,\di t$ with probabilities
$\bar p_\pm=\bar p_\pm(\alpha,C_+/C_-,\beta, A_+/A_-)$, see [Pilipenko and Proske, Stat.\ Probab.\ Lett., 132:62--73, 2018]. 
In this paper we solve the generalized 
selection problem for the stochastic differential equation $\di X^\e=a(X^\e)\,\di t+\e b(X^\e)\,\di Z$ 
whose dynamics in the vicinity of the origin in certain sense
reminds of dynamics of the model equation. In particular we show that solutions $X^\e$ also converge to the 
maximal or minimal solutions of the limiting irregular ordinary 
differential equation $\di x=a(x) \,\di t$ with the same \emph{model} selection probabilities
$\bar p_\pm$. This means that for a large class of irregular stochastic differential equations, the selection dynamics
is completely determined by four local parameters of the drift and the 
jump measure.
\end{abstract}


\smallskip


\tableofcontents

 \section{Introduction, setting, and the main result}

The well known Peano existence theorem \cite[Theorem II.2.1]{Hartman-64} states 
that an ordinary differential equation (ODE) $\di x= a(x)\,\di t$ with a continuous function $a\colon \bR\to\bR$
has a local solution which however may be not unique. A classical example of such
non-uniqueness is given by the non-Lipschitzian $\di x= \sqrt{|x|}\,\di t$ 
which allows for a continuum of solutions starting at $x=0$, namely
$x(t)\equiv 0$, and $x(t)=\frac14(t-t_0)_+^2$, $t_0\geq 0$.

On the contrary, the behaviour of stochastic differential equations (SDE) is often more regular. In particular  
and addition of a noise term allows to obtain unique solutions of SDEs with measurable or irregular coefficients. 
We refer the reader to e.g.\  
\cite{zvonkin74,StrVar,veretennikov1981strong,krylov2005strong} for results on SDEs driven by a Brownian motion.
General results on the existence and uniqueness of SDEs with measurable or irregular coefficients driven by L\'evy processes can be found, e.g.\ in \cite{TanTsuWat74,GihSko-82,Situ05,Priola-12,chen2016uniqueness,priola2018davie,kulik2019weak}.

Consider now an SDE with a drift $a$ and assume that the underlying ODE $\di x=a(x)\,\di t$ has multiple solutions. 
A natural question arises, what happens when the random perturbation vanishes. Heuristically, solutions of the small noise SDE 
should converge to one of the various deterministic solutions and the \emph{selection problem} consists in description of this limit behaviour.

Originally, the selection problem was treated by \cite{BaficoB-82}, where the authors considered the SDE
\ba
X^\e(t)= x+ \int_0^t a(X^\e(s))\,\di s+ \e\int_0^t b(X^\e {(s)})\,\di W(s),
\ea
with a  drift $a$ being not Lipschitz continuous at $x=0$, and a positive Lipschitz continuous diffusion coefficient $b^2$. They showed that under
certain conditions, the limit law $\Law(X^\e|X^\e(0)=0)$ is supported by the deterministic maximal and minimal solutions of the ODE $\di x=a(x)\,\di t$
starting at zero
with the weights $p_\pm$ that can be 
explicitly determined, see \cite[Theorem 4.1]{BaficoB-82}.
\cite{veretennikov1983approximation} proved the uniqueness of the limit in the case of odd continuous concave drift and additive noise.
Recently \cite{delarue2014transition} gave the new proof of the results by \cite{BaficoB-82} for the piece-wise power drift 
\ba
\label{e:bar_a}
\bar a(x)=A_+ x^\beta\bI_{[0,\infty)}(x)- A_-|x|^\beta\bI_{(-\infty,0)}(x)
\ea
with $\beta\in (0,1)$ and $A_\pm>0$
in the case of additive Brownian perturbations. \cite{trevisan2013zero} studied the same equation with $A_\pm=1$ and $\beta\in[0,1)$. 
\cite{krykun2013peano} generalized the results by \cite{BaficoB-82} to It\^o SDEs with positive diffusion coefficient $a$ of locally bounded variation.
\cite{gradinaru2001singular} analyzed large deviations of the laws of $X^\e$ with additive noise and piece-wise power drift 
\eqref{e:bar_a} with $A_\pm=1$, $\beta\in[0,1)$.

Although the results obtained in \cite{BaficoB-82} are very transparent and intuitively understandable, the intrinsic nature of the selection phenomena and 
especially the methods allowing one to derive the \emph{selection probabilities} $p_\pm$ in more general settings are far from being 
completely understood.

Thus, \cite{pilipenko2018perturbations} considered a class of SDEs $\di X^\e=\bar a(X^\e)\,\di t+\e \di B^{(\alpha)}$
with $\beta\in(-1,1)$ driven by $\alpha$-self-similar processes $B^{(\alpha)}$, e.g.\ by a fractional Brownian motion 
or a strictly stable L\'evy process. They showed that under some natural assumptions $X^\e$ also selects the maximal and minimal solutions of the ODE $\di x=\bar a(x)\,\di t$
with some probabilities $p_\pm$, $p_++p_-=1$. Unfortunately these probabilities cannot be always determined explicitly.

In this paper we address the selection problem for a L\'evy driven SDE with multiplicative noise 
\ba
\label{eq:main_small}
X^\e_x( t)= x+ \int_0^t a(X_x^\e(s))\,\di s+ \e\int_0^t b(X_x^\e(s-))\,\di Z(s),\quad t\geq 0,\quad \e\to 0,
\ea
whose drift $a=a(x)$ has an irregular point at $x=0$ but does not have the exact piece-wise power 
form \eqref{e:bar_a}. The small jumps of the driving L\'evy process $Z$ remind of those
of an $\alpha$-stable L\'evy process.
In other words we answer the question whether the selection dynamics are robust w.r.t.\ perturbations of the drift and the noise.  

Let us formulate the precise assumptions.

\smallskip

\noindent\textbf{A}$_Z$:
Let $Z=(Z_t)_{t\geq 0}$ be a L\'evy process without a Gaussian component
and the jump measure $\nu$ such that for some $\alpha\in (1,2)$ and some constants $C_\pm\geq 0$, $C_-+C_+>0$,
\ba
\label{eq:coeff_assumptions}
&\nu([z,+\infty))\sim C_+ z^{-\alpha} l_\nu\Big(\frac{1}{z}\Big),\\
&\nu((-\infty,-z])\sim C_- z^{-\alpha} l_\nu\Big(\frac{1}{z}\Big),\quad z\to +0,\\
\ea
for a positive function $l_\nu$ slowly varying at infinity.

\medskip

\noindent\textbf{A}$_{a}$:
Let $x\mapsto a(x)$ be a real valued continuous function of linear growth such that $a(0)=0$ and that for $\beta\in(0,1)$
\ba
\label{eq:ass_a}
a(x)=  x^\beta L_+(x)\quad \text{for} \quad x>0 \quad \text{and} \quad 
a(x)= - |x|^\beta L_-(|x|)\quad  \text{for} \quad x<0,
\ea
with continuous functions $L_\pm\colon (0,\infty)\to(0,\infty)$ that satisfy
\ba
\label{eq:L_pm}
L_\pm (x) \sim A_\pm  l\Big(\frac{1}{x}\Big)  \quad  \text{as }    x\to +0,
\ea
for a positive function $l$ slowly varying at infinity, and $A_\pm>0$.

\medskip

\noindent
\textbf{A}$_{b}$: 
Let $x\mapsto b(x)$ be a bounded continuous real valued function such that
\ba
b(0)>0.
\ea

\medskip

\noindent
It follows from assumptions \textbf{A}$_{a}$, \textbf{A}$_{b}$ that 
equation \eqref{eq:main_small} has a weak solution, see, e.g.\  Theorem 1 of \S 2 Chapter 5 in \cite{GihSko-82}. 
\begin{rem}
We will see in the main result that the weak limit of the sequence $\{X^\ve\}$
as $\ve\to 0$ is independent of the choice of weak  solution $X^\e$. So, further we assume that
 $X^\ve$ is any weak solution to \eqref{eq:main_small}.
It should be also noticed that the presence of a noise often implies uniqueness of a solution and the strong Markov property, see references above. 
\end{rem}

\medskip

Let us describe solutions of the limit ODE
\ba
\label{e:ode}
X^0_x(t)=x+ \int_0^t a(X^0_x(s))\,\di s.
\ea
Let $A_\pm(\cdot)$ be continuous non-negative strictly increasing functions given by
\ba\label{eq:A+_def}
&A_+(x):= \int_{0+}^x\frac{\di y}{a(y)},\ x> 0,\\
&A_-(x):= \int^x_{0-}\frac{\di y}{a(y)},\ x< 0,\\
&A_\pm(0)=0, 
\ea
and let $A_\pm^{-1}(\cdot)\colon [0,\infty)\to [0,\infty)$ be their inverses. All these functions are well defined 
because of assumption \textbf{A}$_{a}$.
Hence it is immediate to see that for $x\neq 0$
\ba
\label{e:X0+}
X^0_x(t):= A_+^{-1}(A_+(x)+t),\quad x>0,\ t\geq 0,
\ea
and 
\ba
\label{e:X0-}
X^0_x(t):= -A_-^{-1}(A_-(x)+t),\quad  x< 0,\ t\geq 0,
\ea
are unique solutions of the equation \eqref{e:ode}. 
For $x=0$, there is a continuum of solutions and any solution either has the form 
\ba
\label{eq:solutionODE}
X^0_\pm (t; t_0)=\begin{cases}
         0,\quad t\in[0,t_0),\\
        \pm A_\pm^{-1}(t-t_0), \quad t\in [t_0,\infty),
         \end{cases}
\ea
where $t_0\in[0,+\infty)$ or is trivial $X^0(t)\equiv 0$. 
Among the solutions \eqref{eq:solutionODE} we single out the maximal and the minimal solutions 
\ba
\label{e:eq_pm}
x^\pm(t):= X^0_\pm (t; 0)=\pm A_\pm^{-1}(t), \quad t\geq 0.
\ea
It is intuitively clear that any solution $X^\e$ starting at zero should select one of the particular solutions $x^\pm$ of \eqref{e:ode}
very quickly, so that one can expect that the selection is determined only by 
the small jumps of $Z$ and the local behaviour of $a(\cdot)$ and $b(\cdot)$ in the vicinity
of zero. Taking into account assumptions  \textbf{A}$_{Z}$ and \textbf{A}$_{a}$
we introduce the auxiliary \emph{model} SDE
\ba
\label{e:model}
\bar X^\e(t)=\int_0^t \bar a(\bar X^\e(s))\,\di s+\e Z^{(\alpha)}(t)
\ea
with the piece-wise power drift $\bar a$ defined in \eqref{e:bar_a} and   
driven by a zero mean strictly $\alpha$-stable L\'evy process $Z^{(\alpha)}$, $\alpha\in (1,2)$,
with the characteristic function
\ba\label{eq:char_stable}
&\ln \E \ex^{\i \lambda Z^{(\alpha)}(1)}= \int (\ex^{\i \lambda z}-1-\i \lambda z)\nu^{(\alpha)}(\di z),\quad \lambda\in\bR,\\
\ea
and the jump measure
\ba
\label{eq:nu_alpha}
\nu^{(\alpha)}([z,+\infty))= C_+ z^{-\alpha},  \quad
\nu^{(\alpha)}((-\infty,-z])= C_- z^{-\alpha},\ z>0 .
\ea
The model equation \eqref{e:model} has a unique strongly Markovian solution due to Theorem 3.1
from \cite{TanTsuWat74} (although in \cite{TanTsuWat74} the drift is supposed to be bounded, an extension of their results
to $\bar a$ given by \eqref{e:bar_a} follows 
easily from the sublinear growth of $\bar a$ at infinity).

The \emph{model} ODE $\di x=\bar a(x)\,\di t$ has the following maximal and minimal solutions starting at $x=0$:
\be
\label{e:barx}
\bar x^\pm(t)=\pm\big( A_\pm(1-\beta)t\big)^{\frac{1}{1-\beta}},\quad t\geq 0.
\ee

The selection problem 
for the model SDE \eqref{e:model} was solved by \cite{pilipenko2018perturbations}.
 
\begin{thm}[\cite{pilipenko2018perturbations}]
\label{thm:PP}
Let $\bar X^\ve$ be a solution to the model equation 
\eqref{e:model}.
 Then 

1) for any $\ve>0$  
\ba
\label{eq:Xto infty}
&\P\Big(\lim_{t\to\infty }|\bar X^\ve(t)|=+\infty\Big)=1
\ea
and the probabilities
\ba
\label{eq:p+-}
&\bar p_\pm=\P\Big(\lim_{t\to\infty }\bar X^\ve(t)=\pm \infty\Big)
\ea
are independent of $\ve$, and   $\bar p_-+ \bar p_+=1$;

2) the convergence 
\ba
\Law \bar X^\e \Rightarrow \bar  p_- \delta_{\bar x^-}+  \bar p_+\delta_{\bar x^+} ,\quad\e\to 0,
\ea
in $D([0,\infty),\bR)$ holds true, where $\bar x^\pm$ are defined in \eqref{e:barx}. 
\end{thm}

\begin{rem}
If $\alpha=2$, i.e.\ if $Z^{(\alpha)}$ is a Brownian motion
then the probabilities $\bar p_\pm$ are known explicitly:
\ba
\bar p_-=\frac{A_+^{-\frac{1}{1+\beta}}}{A_-^{-\frac{1}{1+\beta}}+A_+^{-\frac{1}{1+\beta}}}\quad\text{ and }\quad
\bar p_+=\frac{A_-^{-\frac{1}{1+\beta}}}{A_-^{-\frac{1}{1+\beta}}+A_+^{-\frac{1}{1+\beta}}},
\ea
see \cite{BaficoB-82,delarue2014transition}.
\end{rem}

\begin{rem}
\label{rem:p_indep_of_sigma}
It follows from the self-similarity of $Z^{(\alpha)}$
that for any $\e,\delta,\gamma>0$ the rescaled process $\bar X^{\gamma,\delta,\ve}(t):=\gamma \bar X^\ve(\delta t)$, $t\geq 0$, satisfies the SDE
\ba
\bar X^{\gamma,\delta,\ve}(t)=\int_0^t \gamma^{1-\beta}\delta\; \bar a(\bar X^{\gamma,\delta,\ve}(s))\,\di s
+\e  \gamma  \delta^{\frac{1}{\alpha}}\cdot \bar Z^{(\alpha)}(t),
\ea
where $\bar Z^{(\alpha)}\stackrel{\di}{=} Z^{(\alpha)}$. 
This implies that the selection probabilities $\bar p_\pm$ defined in \eqref{eq:p+-}
are the same for any model equation
 \ba
 \label{e:model1}
 \bar X^\e(t)=\int_0^t \bar a(\bar X^\e(s))\,\di s+\e\cdot \sigma\cdot Z^{(\alpha)}(t)
 \ea
 with any $\sigma>0$. 
Moreover, they are completely determined by the four parameters $\alpha\in(1,2)$, $C_+/C_-\in[0,+\infty]$,
$\beta\in (0,1)$, and $A_+/A_-\in (0,\infty)$.
\end{rem}

In the present paper we solve the generalized selection problem for the SDE
 \eqref{eq:main_small}.  
 The main result of this paper is the following.

\begin{thm}
\label{thm:main}
Let assumptions \emph{\textbf{A}}$_{Z}$, \emph{\textbf{A}}$_{a}$, and \emph{\textbf{A}}$_{b}$
hold true, and let
$X^\e$ be a solution to \eqref{eq:main_small} with the initial condition $x=0$, namely
\ba
\label{e:X}
X^\e(t)=\int_0^t a(X^\e(s))\,\di s+\e\int_0^t b(X^\e(s-))\,\di Z(s),\quad t\geq 0.
\ea
Then 
\ba
\label{eq:main_conv}
\Law X^\e\Rightarrow \bar p_- \delta_{x^-}+\bar p_+\delta_{x^+},\quad \e\to 0,
\ea
in $D([0,\infty),\bR)$ where functions $x^\pm$ are defined in \eqref{e:eq_pm} and
the selection probabilities $\bar p_\pm$ are determined in Theorem \ref{thm:PP}
 for the model equation \eqref{e:model}.
\end{thm}

Before proceeding with the proofs we give several clarifying remarks.

\begin{rem}
Theorem \ref{thm:main} states that the generalized selection probabilities of the equation \eqref{e:X} 
coincide with the selection probabilities $\bar p_\pm$ of the 
model equation \eqref{e:model}.
Hence the selection behaviour is robust 
with respect to appropriate perturbations of a) the drift, b) the L\'evy measure in the vicinity of the origin, and c)
with respect to incorporation of the multiplicative noise. Essentially, the selection probabilities for the whole class of SDEs 
\eqref{e:X} depend only on the four parameters of the model equation.

Our results agree with the results by \cite{BaficoB-82} for Gaussian diffusions ($\alpha=2$) where $\bar p_\pm$ were determined
in terms of certain integrals of $a(x)/b(x)$, see Eq.\ (3.4) and Theorem 4.1 in  \cite{BaficoB-82}.
\end{rem}

\begin{rem}
\label{rem:246}
If $x\neq 0$, then it is easy to verify that 
\ba
\Law X^\ve_x \Rightarrow \delta_{X^0_x}, \quad \ve\to 0,
\ea
with $X^0_x$ defined in \eqref{e:X0+} and \eqref{e:X0-}.
\end{rem}

\begin{rem}
We emphasize that although we do not assume uniqueness of 
(weak) solutions $X^\e$  of \eqref{e:X} for $\e\geq 0$, the weak limit  \eqref{eq:main_conv}
is unique.
\end{rem}

\begin{rem}
The question how to determine the selection probabilities $\bar p_\pm$ is still open. Although the results by \cite{pilipenko2018perturbations}
establish the existence of $\bar p_\pm$ for the model equation for self-similar noises it is clear that quite different methods should be used for
SDEs driven by L\'evy processes or, say, a by a fractional Brownian motion.  
\end{rem}

\begin{rem}
Eventually we note that Theorem \ref{thm:main} gives us the existence and uniqueness of the weak limit. There is a number of works 
in which pathwise restoration of uniqueness for ODEs with an irregular or even distributional drift $a$ by adding a random perturbaton is studied.
For example the regularization by adding a sample Brownian path was studied by \cite{davie2007uniqueness,davie2011,Flandoli-2011,flandoli2011regularizing,shaposhnikov2016some,AlabLeon-17,banos2018construction,banos2019restoration}.
The same problem for the fractional Brownian motion was treated by \cite{banos2015strong,catellier2016averaging,barrimi2017approximation,banos2017c,harangperkowski20,galeatigubinelli20}.
\end{rem}

\begin{rem}
The selection problem in a multidimensional setting was also tackled recently by \cite{pilipenko2018selection} and \cite{delarue2019zero}.
Small noise behaviour of multidimensional SDEs with discontinuous drift was also studied by \cite{buckdahn2009limiting} in the setting of
differential inclusions. 
\end{rem}

The rest of the paper is devoted to the proof of the main result. To make the arguments more transparent we preface the proof with 
a heuristic description of the steps and explain the structure of the paper.

First we consider the process $X^\e$ and note that due to the boundedness of $b$ and the sublinear growth and the continuity of the drift $a$, 
the family of distributions $\{\Law(X^\e)\}_{\e\in (0,1]}$ 
is tight
in $D([0,\infty),\bR)$ and any (weak) limit point is a solution of the ODE \eqref{e:ode} with $x=0$.
All possible solutions to  
\eqref{e:ode}  have been described in \eqref{eq:solutionODE}.

To prove Theorem \ref{thm:main} it suffices to show two properties of the limit laws of $X^\e$ as $\e\to 0$. 
First, a process $X^\e$ can spend only infinitesimal time near zero and hence
it chooses either the maximal or the  minimal solutions $x^\pm$ of the ODE \eqref{e:ode}; in other words, no solution 
$X_\pm^0(\cdot; t_0)$ with $t_0>0$ (see \eqref{eq:solutionODE}) can support the limiting law of $X^\e$. Second,
the deterministic solutions $x^\pm$ should be chosen with the probabilities $\bar p_\pm$ determined in \eqref{eq:p+-}.

Hence we will show that the selection takes place with probabilities $\bar p_\pm$ in an infinitesimal time-space 
box $t\in[0,T_0\e']$, $x\in[-R\e'',R\e'']$ with appropriately chosen bounds $\e'=\e'(\e)\to 0$ and $\e''=\e''(\e)\to 0$ and $T_0>0$, $R>0$ large enough.
To achieve this, we introduce a rescaled process $Y^\e(t):=X^\e(\e' t)/\e''$ and show that $Y^\e$ converges weakly to a solution
of the model equation \eqref{e:model1} with $\sigma=b(0)$. Hence the exit of $X^\e$ from the infinitesimal time-space 
box $[0,T_0\e']\times [-R\e'',R\e'']$ is equivalent to the exit of $Y^\e$ from the $\e$-independent time-space box $[0,T_0]\times [0,R]$ which is controlled 
by Theorem \ref{thm:PP}

The second step is to show that upon leaving the $\e$-dependent time-space box $[0,T_0\e']\times [-R\e'',R\e'']$ with $R>0$ sufficienty large, 
a solution $X^\e$ with high probability follows the maximal (minimal) solution $x^\pm$ as $\e\to 0$. 
Here it suffices to construct a deterministic increasing (decreasing) function that bounds $X^\e$ from below (above) with high probability.

The paper is organized as follows. In Section \ref{s:ts} we chose the appropriate scales $\e'$ and $\e''$ and show that
the rescaled process $Z_\e(t)=Z(\e' t)/\e''$ converges to the $\alpha$-stable process $Z^{(\alpha)}$ defined in \eqref{eq:char_stable}
and the rescaled process $Y^\e$ converges to the solution of the $\e$-independent model equation. 
In Section \ref{s:noise_estim} we obtain algebraic growth rates of 
the noise term  $\int_0^\cdot b(X^\ve(s-))\,\di Z(s)$ that are uniform over 
$\e\in(0,1]$ and the initial value $X^\e(0)$.
In Section \ref{s:exit} we study the exit of $X^\e$ from the time-space box $[0, T_0\e']\times [-R\e'',R\e'']$.
In Section \ref{s:exit2} we determine deterministic lower and upper bounds that push 
a solution $X^\e$ with an initial value $|X^\e(0)|\geq R\e''$ away from zero with high probability.
This will finish 
the proof of Theorem \ref{thm:main}. 

\medskip
\noindent
\textbf{Acknowledgements.} This research was partially supported by the Alexander von Humboldt Foundation within 
the Research Group Linkage Programme {\it Singular diffusions: analytic and stochastic approaches} between the University of Potsdam and the Institute of Mathematics of the National Academy of Sciences of Ukraine. A.P.\ thanks the Friedrich Schiller University of Jena for hospitality.

\section{Preliminary considerations and time-space rescaling\label{s:ts}}

Before starting the proof we make two technical assumptions that do not reduce the generality of the 
setting but simplify the arguments significantly.

\begin{rem}\label{rem:trunc_tails}
To establish convergence \eqref{eq:main_conv}  it suffices to show the weak
convergence on the space $D([0,T],\bR)$ for each $T>0$.
We will use the truncation of large jumps procedure.

For $M>0$, let  
\ba
Z^M(t)=Z(t)-\sum_{s\leq t} \Delta Z(s)\cdot \bI(|\Delta Z(s)|>M)
\ea
be the L\'evy process with bounded jumps.
For each $T>0$ and $\theta>0$ we can find $M>0$ large enough such that
\ba
\P\Big(Z(t)=Z(t)^M, t\in[0,T]\Big)=1-\exp\Big(-T\int_{|z|>M}\nu(\di z)\Big)\geq
1-\theta.
\ea
Then for any solution $X^\e$
there exists a solution $X^{\e,M}$ of \eqref{eq:main_small} driven by the
process $Z^M$
such that
\ba
\P\Big(X(t)^\e=X(t)^{\e,M},\, t\in[0,T]\Big)\geq 1-\theta
\ea
(we consider all processes on an appropriate probability space).
Hence in order to prove weak convergence of the processes $X^\e$ it is
sufficient to prove convergence for
the processes $X^{\e,M}$ under the additional assumption that for some
$M>0$
\ba
\label{eq:fin_tail}
\operatorname{supp} \nu\subseteq [-M,M]\quad \text{and}\quad \nu(\{\pm M\})=0.
\ea 
From now on we assume \eqref{eq:fin_tail} to hold for the process $Z$.
\end{rem}
\begin{rem}
\label{r:truncate}
Similarly to the previous remark we also note that for any two drifts $a$ and $\tilde a$ both satisfying \textbf{A}$_{a}$ and such that
$a(x)=\tilde a(x)$, $|x|\leq 1$, the corresponding solutions $X^\e$ and $\tilde X^\e$ coincide up to the exit from $[-1,1]$. 
Hence the selection probabilities
for these solutions in the limit $\e\to 0$ are equal too. From now on we assume without loss of generality that
\ba\label{eq:trunc_drift}
L_\pm(x)=L_\pm(x\wedge 1) \text{ for }   x>0
\ea 
to ensure the power growth of $a$ at infinity.
\end{rem}

\begin{lem}
\label{l:X}
Assume that assumptions  \emph{\textbf{A}}$_{a}$ and \emph{\textbf{A}}$_{b}$ 
are satisfied. Then the family of distributions $\{\Law(X^\e)\}_{\e\in(0,1]}$
is tight in $D([0,\infty),\bR)$ and a limit $X$ of any
weakly convergent subsequence $\{X^{\ve_n}\}_{n\geq 1}$, $X^{\ve_n}\Rightarrow X$, $n\to\infty$,
satisfies the integral equation
\ba
X(t)= \int_0^t a(X(s))\,\di s.
\ea
\end{lem}
\begin{proof}
First we note that since $b$ is bounded,
\ba
\label{e:bZ}
\e\int_0^t b(X^\e(s-))\,\di Z(s)\Rightarrow 0,\quad \e\to 0,
\ea
weakly in the uniform topology. 
Tightness of $\{\Law(X^\e)\}_{\e\in(0,1]}$ follows, e.g.\ from the continuity of $a$, Aldous' criterion and 
boundedness
\ba
\sup_{\e\in(0,1]} \E \sup_{t\in[0,T]}|X_t^\e|^2\leq C(T)<\infty
\ea
for each $T>0$ and some $C(T)>0$.

Finally, due to the continuity of $a$, 
for any weakly convergent subsequence $X^{\ve_n}\Rightarrow X$ we get the weak convergence of the pairs
\ba
\Big(X^{\e_n}(\cdot),  \int_0^\cdot a(X^{\e_n}(s))\,\di s \Big)\Rightarrow  
\Big(X(\cdot),  \int_0^\cdot a(X(s))\,\di s \Big), \ n\to\infty,
\ea
in $D([0,\infty),\bR)\times C([0,\infty),\bR)$
which together with \eqref{e:bZ} implies the result.
\end{proof}

Let $X^\e$ be any solution of \eqref{e:X}. For any $\e'=\e'(\e)>0$ and $\e''=\e''(\e)>0$ consider a time-space rescaled process
\ba
\label{eq:Y_rescaling}
Y^\e(t)=\frac{X^\e(\e' t)}{\e''},\quad t\geq 0,
\ea
which satisfies the SDE
\ba
\label{e:Ye}
Y^\e(t)=\frac{X^\e(\e' t)}{\e''}
&=\frac{1}{\e''}\int_0^{\e' t} a(X^\e(s))\,\di s+  \frac{\e}{\e''}
\int_0^{\e' t} b(X^\e {(s-)})\,\di Z(s)\\
&=\int_0^t \frac{a(\e'' Y^\e(s))}{\e''/\e'}\,\di s+
 \int_0^t b(\e'' Y^\e(s-))\,\di \frac{Z(\e' s)}{\e''/\e}\\
&=\int_0^t  {a_\e( Y^\e(s))}\,\di s+ \int_0^t b_\e( Y^\e(s-))\,\di
{Z_\e( s)},
\ea
where
\ba\label{eq:a_e,b_e}
a_\e(y)=\frac{a(\e'' y)}{\e''/\e'}, \quad  b_\e(y)=  b(\e'' y),\quad
Z_\e(t)= \frac{Z(\e' t)}{\e''/\e}.
\ea

\begin{lem}
\label{l:eee}
There exist positive null sequences $\e'=\e'(\e)$ and  $\e''=\e''(\e)$ 
such that 
\ba
\label{e:eee0} 
\lim_{\e\to 0}\frac{\e''}{\e} = 0
\ea
and
\begin{align}
\label{e:eee1}
&\frac{\e''}{\e'}\sim  (\e'')^\beta l\Big(\frac{1}{\e''}\Big),\\
\label{e:eee2}
&\e'\sim  \Big(\frac{\e''}{\e}\Big)^{\alpha} \cdot \Big( l_\nu\Big(\frac{\e}{\e''}\Big)\Big)^{-1} \quad  \text{as} \quad \e\to 0.
\end{align}
\end{lem}
\begin{proof} 
Recall that a product, a sum, and a ratio of two positive slowly varying functions is again a slowly 
varying function (Proposition 1.3.6 in \cite{BinghamGT-87}). 
Furthermore
due to Theorem 1.5.12 from \cite{BinghamGT-87}, each regularly varying function $f$ with  index $\gamma>0$
has an asymptotic inverse function $g$ that is regularly varying with index $1/\gamma$, namely
\ba\label{eq:gen_inv_RV}
f(g(x))\sim g(f(x)) \sim x,\ \ 	x\to\infty.
\ea
Consider functions $f_1(x)=x^{1-\beta}/l(x)$
and $f_2(x)=x^{\alpha} l_\nu(x)$, $x>0$, that are regularly varying at infinity and let 
\ba
\label{e:g1g2}
g_1(x)=x^{\frac{1}{1-\beta}}l_1(x)\quad \text{and}\quad g_2(x)=x^{\frac{1}{\alpha}}l_2(x)
\ea
be their asymptotic inverses,
where $l_1$ and $l_2$ are slowly varying at infinity functions. Since $\frac1{1-\beta}-\frac1{\alpha}>0$, the function
\ba
f_3(x)=x^{\frac1{1-\beta}-\frac1{\alpha}}\frac{l_1(x)}{l_2(x)}
\ea
is also regularly varying with positive index. Let
\ba
g_3(x)=x^{(\frac1{1-\beta}-\frac1{\alpha})^{-1}}l_3(x)
\ea
be its asymptotic inverse.

We set
\ba
&\e'(\ve):= \ve^{({\frac1{1-\beta}-\frac1{\alpha}})^{-1}} l_3\Big(\frac{1}{\ve}\Big)^{-1}= \frac{1}{g_3(\frac{1}{\e})},\\
&\e''(\ve):=  (\ve'(\ve))^{\frac1{1-\beta}} l_1\Big(\frac{1}{\e'(\e)}\Big)^{-1}= \frac{1}{g_1(\frac{1}{\e'})}.
\ea
It is easy to see that $\e\mapsto \e'(\e)$ and $\e\mapsto\e''(\ve)$ satisfy conditions of the Lemma. A straightforward verification yields the 
equivalence \eqref{e:eee1}:
\ba
f_1\Big(\frac{1}{\e''}\Big)&= f_1 \Big(g_1\Big(\frac{1}{\e'}\Big)\Big)
\sim \frac{1}{\e'}.
\ea
Furthermore, since $\frac{1}{\e}\sim f_3(\frac{1}{\e'})$ we get
\ba
\label{eq:727}
\frac{\e}{\e''}\sim \frac{g_1(\frac{1}{\e'})}{f_3(\frac{1}{\e'})}=g_2\Big(\frac{1}{\e'}\Big).
\ea
Since any regularly varying function preserves equivalence, see \cite[Theorem 3.42]{buldygin2018pseudo},
we obtain \eqref{e:eee2} by application of $f_2$ to \eqref{eq:727}:
\ba
f_2\Big(\frac{\e}{\e''}\Big)\sim f_2\Big(g_2\Big(\frac{1}{\e'}\Big)\Big)\sim \frac{1}{\e'}.
\ea
\end{proof}

Let $\nu$ be the L\'evy measure of the process $Z$ satisfying \textbf{A}$_{Z}$ and \eqref{eq:fin_tail}, and let $\e'$, $\e''$ be the 
sequences chosen in Lemma \ref{l:eee}. For $\e\in(0,1]$ let us define rescaled jump measures $\nu_\e$ by setting
\ba
\label{e:nueps}
\nu_\e([ z,\infty ))& =\e' \nu\Big(\Big[ \frac{\e''z}{\e},\infty \Big)\Big),\\
\nu_\e((-\infty,-z])&=
\e' \nu\Big(\Big(-\infty, -\frac{\e'' z}{\e}\Big]\Big),\quad z>0.
\ea

\begin{lem}
\label{l:lm}
For the family of jump measures $\{\nu_\e\}_{\e\in(0,1]}$ defined in \eqref{e:nueps} we have: \\
1. for each $z>0$
\ba
\label{e:nuepsconv}
\lim_{\e\to 0} \nu_\e([ z,\infty ))&=\nu^{(\alpha)}([z,\infty)),\\
\lim_{\e\to 0} \nu_\e((-\infty,-z])&=\nu^{(\alpha)}((-\infty,z]),
\ea 
where $\nu^{(\alpha)}$ is defined in \eqref{eq:nu_alpha}.

\noindent
2. for each $\delta>0$ there is $C>0$ such that for all $z>0$ 
\ba
\label{eq:Levy_lim1}
\sup_{\e\in(0,1]}\Big(\nu_\e ((-\infty,-z])+  \nu_\e ([z,\infty)) \Big)\leq C \Big(\frac{1}{z^{\alpha-\delta}} \vee \frac{1}{z^{\alpha+\delta}}\Big).
\ea
\end{lem}
\begin{proof}
Without loss of generality we consider only the right tail of $\nu_\e$.   

\noindent
1. For any $z>0$ we apply \eqref{eq:coeff_assumptions}, \eqref{e:eee0} and \eqref{e:eee2} to get for $\e\to 0$ that
\ba 
\label{eq:Levy_measure_transform}
\nu_\e([ z,\infty ))& =\e' \nu\Big(\Big[ \frac{\e''z}{\e},\infty \Big)\Big)
\sim C_+\e'\Big(\frac{\e''z}{\e}\Big)^{-\alpha}l_\nu\Big(\frac{\e}{\e'' z}\Big)
\sim \frac{C_+}{z^\alpha}\cdot\frac{l_\nu\big(\frac{\e}{\e'' z}\big)}{l_\nu\big(\frac{\e}{\e''}\big)} 
\sim \frac{C_+}{z^\alpha}=\nu^{(\alpha)}([z,\infty)).
\ea
2. Let $\delta>0$, $\e\in (0,1]$, $z>0$. We consider two cases. First, let $0<\frac{\e''z}{\e}\leq M$. Then we take into account \eqref{e:eee2} and \eqref{eq:coeff_assumptions} and 
apply Potter's theorem, see e.g.\ \cite[Theorem 1.5.6]{BinghamGT-87} to get
\ba
 \nu_\e ([z,\infty))=\e' \nu\Big(\Big[ \frac{\e''z}{\e},\infty \Big)\Big)
&=\frac{\e'}{(\frac{\e''}{\e})^{\alpha}l_\nu(\frac{\e}{\e''} )^{-1} } \cdot\frac{\nu([ \frac{\e''z}{\e},\infty  ) )}{C_+ (\frac{\e''z}{\e})^{-\alpha}l_\nu(\frac{\e}{\e''z} )  }
\cdot \frac{C_+ (\frac{\e''z}{\e})^{-\alpha}l_\nu(\frac{\e}{\e''z} )  }{  (\frac{\e''}{\e})^{-\alpha}l_\nu(\frac{\e}{\e''} )  }\\
&\leq \sup_{\e\in(0,1]}\frac{\e'}{(\frac{\e''}{\e})^{\alpha}l_\nu(\frac{\e}{\e''} )^{-1} } \cdot
\sup_{y\in(0,M]}\frac{\nu([ y,\infty  ) )}{C_+ y^{-\alpha}l_\nu(\frac{1}{y} )  }
\cdot \frac{C_+}{z^\alpha}\cdot (z^{-\delta}\vee z^\delta)\\
&=\frac{C(\delta,M)}{z^{\alpha}}\cdot (z^{-\delta}\vee z^\delta).
\ea
Second, for $\frac{\e''z}{\e} > M$ by \eqref{eq:fin_tail} we have
\ba
\nu_\e([z,\infty))=0.
\ea
\end{proof}

\begin{thm}
\label{t:ZY}
Suppose that $\e'$ and $\e''$ satisfy \eqref{e:eee1} and \eqref{e:eee2}, and assumptions of Theorem \ref{thm:main} hold true. 
Then  \\
1.
\ba
\label{e:ZZ}
Z_\e \Rightarrow Z^{(\alpha)},\quad \e\to 0,
\ea
where $Z^{(\alpha)}$ is defined in \eqref{eq:char_stable};\\
2.  
there exists a weak limit 
\ba
\label{eq:conv_Y}
Y^\e\Rightarrow Y,\quad \e\to 0,
\ea
which satisfies the SDE
\ba
\label{e:eqY}
Y(t)=\int_0^t \bar a(Y(s)) \,\di s+ b(0)Z^{(\alpha)}(t),\quad t\geq 0.
\ea
The process $Y$ diverges to $\pm\infty$ with the selection probabilities $\bar p_\pm$ defined in Theorem \ref{thm:PP}.
\end{thm}

\begin{proof}

\noindent 
1. It is well known that in the case of L\'evy processes convergence of
marginal
 distributions implies the weak convergence in the Skorokhod space,
see \cite[Corollary VII.3.6]{JacodS-03}.

For some $\mu\in\bR$, the process $Z$ has the L\'evy--Khintchine
representation 
\ba
\ln\E \ex^{\i \lambda Z(1)}
=\i\mu\lambda +  \int_\bR \big(\ex^{\i \lambda  z} -1 - \i\lambda  z\big)\nu(\di z),\quad \lambda\in\bR,
\ea
whereas the rescaled process $Z_\e$
has the L\'evy--Khintchine representation
\ba
\ln \E \ex^{\i \lambda Z_\e(1)}=\ln \E \ex^{\i \frac{   \lambda  Z(\e') }{\e''/\e} 
}
&
=  \frac{\i\mu \lambda  \e'  }{\e''/\e}
+\e' \int_\mbR \Big(\ex^{\i \frac{   \lambda z   }{\e''/\e}} -1 - \frac{  \i  \lambda z   }{\e''/\e} \Big)\nu(\di z) \\
&=
\i \mu_\e \lambda+ 
\int_\mbR (\ex^{\i \lambda z} -1 -  {\i \lambda  z } ) \nu_\e(\di z),
\ea
with the jump measures $\nu_\e$ defined in \eqref{e:nueps}.

Hence, the integration by parts formula, Lebesgue's dominated
convergence theorem, \eqref{e:nuepsconv} and \eqref{eq:Levy_lim1} yield that for each $\lambda\in\bR$
\ba
\label{eq:601}
\int_{(0,\infty)} (\ex^{\i \lambda z} -1 -  {\i \lambda  z } ) \nu_\e(\di z)
&= -\i\lambda \int_{(0,\infty)}  (\ex^{\i \lambda z} -1  ) \nu_\e([z,\infty))\,
\di z\\
&\to - \i\lambda \int_{(0,\infty)}  (\ex^{\i \lambda z} -1  )
\nu^{(\alpha)}([z,\infty))\, \di z\\
&= \int_{(0,\infty)}  (\ex^{\i \lambda z} -1 -  {\i \lambda  z } )
\nu^{(\alpha)}(\di z),\quad \e\to 0.
\ea
The same convergence holds analogously for the negative tail.

Eventually it follows from the choice of $\e'$ and $\e''$ (see \eqref{e:g1g2} and \eqref{eq:727}) that 
\ba
\mu_\e=\mu\cdot \frac{\e\cdot \e'}{\e''}\sim \mu\cdot (\e')^{\frac{\alpha-1}{\alpha}}l_2\Big(\frac{1}{\e'}\Big)\to 0,\quad \e\to 0.
\ea
Therefore we obtain convergence
of the characteristic functions
\ba
\E \ex^{\i \lambda Z_\e(1)}\to \E \ex^{\i \lambda Z^{(\alpha)}(1)}, \
\e\to 0.
\ea
2. To show \eqref{eq:conv_Y}, first we note that for $\bar a$ defined in \eqref{e:bar_a} and $b\in C_b(\bR,\bR)$ the convergence
\ba
\label{eq:lim_coeff}
\lim_{\e\to 0}{a_\e(y)}=\bar a(y)\quad\text{and}\quad
\lim_{\e\to0}b_\e(y)=b(0),
\ea
holds point-wise and uniformly on compact intervals.
To prove that solutions $Y^\e$ converge to $Y$ we follow the standard two-step scheme that consists in showing the tightness of the family $\{Y^\e\}$ 
and the identification of the limit.

To show tightness, one mimics the arguments of \S 2 of Chapter 5 of \cite{GihSko-82}. Indeed, one shows that $Y^\e$ are bounded in probability 
on compact time intervals which together with the linear growth of $a$ implies the weak compactness of the integrals $\int_0^\cdot a_\e(Y^\e(s))\,\di s$. 
The weak compactness of the noise term $\int_0^\cdot b_\e(Y^\e(s))\,\di Z_\e(s)$ follows from the boundedness of $b_\e$ and the weak convergence 
\eqref{e:ZZ}. 

Eventually the identification of the limit is obtained with the help of Theorem IX.4.8 from \cite{JacodS-03}. 

Due to Theorem \ref{thm:PP} and Remark \ref{rem:p_indep_of_sigma}, the limiting process $Y$ diverges to
$\pm\infty$ with the selection probabilities $\bar p_\pm$.
\end{proof}

\section{Estimates for the noise\label{s:noise_estim}}

In this section we get estimates for a growth rate  of 
the noise term $\int_0^t  b_\e(Y^\e(s-))\,\di Z_\e(s)$ as $t\to\infty$
that are uniform in $\e$.
We start with the the following general result.

\begin{lem}
\label{lem:sup_Levy} 
Let $\tilde Z$ be a zero mean L\'evy process without a Gaussian component
and with a jump measure $\nu$
such that for some $C>0$ and $ \gamma\in (1,2)$ it satisfies
\ba
\label{e:Cx}
\int_{|z|> x}\nu(\di z) \leq \frac{C}{x^\gamma}, \quad x\geq 1.
\ea
and
\ba
\label{e:Cz}
\int_{|z|\leq 1} z^2\nu(\di z) \leq C.
\ea
Then for any $\theta>0$ and $\delta>0$ there exists a generic constant $K=K(C,\gamma,\delta,\theta)$
such that for any predictable process $\{\sigma(t)\}_{t\geq 0}$, $|\sigma(t)|\leq 1$ a.s., we have
\ba
\label{e:noise}
\P \Big(\sup_{t\geq 0} \frac{\int_0^t  \sigma(s)\,\di \tilde Z(s)  }{ 1+t^{\frac{1}{\gamma}+\delta} } \leq K\Big)
\geq 1-\theta.
\ea
\end{lem}
\begin{proof} 
Denote $T(x):=\nu((-x,x)^c)$, $x\geq 1$. With the help of the integration by parts and \eqref{e:Cx} we get for $x\geq 1$ that
\ba
\label{e:est1}
\int_{|z|>x} |z|\, \nu(\di z) =-\int_x^\infty z\, \di T(z)=-zT(z)\Big|_x^\infty+ \int_x^\infty T(z)\,\di z
\leq C  x^{1-\gamma}+\frac{C}{\gamma-1}x^{1-\gamma}=\frac{C\gamma}{\gamma-1}x^{1-\gamma}.
\ea
Furthermore, for $x\geq 1$
\ba
\label{e:est2}
\int_{0< |z|\leq x} z^2\, \nu(\di z) 
&\leq \int_{0< |z|\leq 1} z^2\, \nu(\di z)   +   \int_{1< |z|\leq x} z^2\, \nu(\di z) 
\leq C-z^2 T(z)\Big|_1^x+ 2\int_1^x z T(z)\,\di z
\\
&\leq 2C + \frac{2C}{2-\gamma}x^{2-\gamma}  \leq   2C\frac{3-\gamma}{2-\gamma} x^{2-\gamma} .
\ea

To show \eqref{e:noise}, we follow the reasoning by \cite{pruitt1981growth}. Let us use the L\'evy--It\^o representation of the process $\tilde Z$,
namely for a Poissonian random measure $N$ with the compensator $\nu(\di z)\di t$ we write
\ba
\tilde Z(t)=\int_0^t\int z\,\tilde N(\di z,\di s).
\ea
For arbitrary $A\geq 1$ and $T>0$ we estimate
\ba
\label{eq:sup_int}
\P\Big(\sup_{t\in [0,T]} &\Big|{\int_0^t  \sigma(s)\,\di \tilde Z(s)  }\Big|> A\Big)
\leq\P\Big(\sup_{t\in [0,T]} \Big|\int_0^t  \sigma(s) \int_{|z|\leq  A} z \tilde N(\di z, \di s)  \Big|>  \frac{A}{3}\Big)\\
&+ \P \Big(\int_0^T  \int_{|z|> A}   N(\di z, \di s)  >0\Big)
+\P \Big(  \int_0^T  \int_{  |z|> A}  |z| \nu(\di z) \di s  >  \frac{A}{3}\Big)=I_1+I_2+I_3.
\ea
By Doob's inequality and \eqref{e:est2} we obtain
\ba
\label{eq:1040}
I_1\leq \frac{36\int_0^T \E\sigma^2(s) \int_{|z|\leq A} z^2 \nu(\di z) \di s}{A^2}\leq
 \frac{ 36T  \int_{|z|\leq A} z^2 \nu(\di z)  }{A^2}\leq  \frac{3-\gamma  }{2-\gamma}\frac{72 C T}{A^{\gamma}}.
\ea
The inequality $1-\ex^{-x}\leq x$, $x\geq 0$, and \eqref{e:Cx} imply that
\ba\label{eq:1041}
I_2=1-\exp\Big(-T\int_{|z|>A}\nu(\di z)\Big)\leq T\int_{|z|>A}\nu(\di z)\leq \frac{ C T  }{A^{\gamma}}.
\ea
The item $I_3$ equals 0 if $T \int_{ |z|> A}  |z|\, \nu(\di z)\, \di s  \leq  A/3$.
By \eqref{e:est1} this is true if $3C T \frac{\gamma}{\gamma-1}A^{-\gamma}\leq 1$.

Hence for each $K>0$ we have
\ba
\P \Big(\sup_{t\geq 0} \frac{\int_0^t  \sigma(s)\,\di \tilde Z(s)  }{ 1+t^{\frac{1}{\gamma}+\delta} } &>  K\Big)
\leq  \P \Big(\sup_{t\in[0,1]}  {\int_0^t  \sigma(s)\,\di \tilde Z(s)  }  >K\Big)
+\sum_{n=0}^\infty
\P \Big(\sup_{t\in[2^n, 2^{n+1}]} \frac{\int_0^t  \sigma(s)\,\di \tilde Z(s)  }{ 1+t^{\frac{1}{\gamma}+\delta} } >  K\Big)\\
&\leq 
 \P \Big(\sup_{t\in[0,1]}  {\int_0^t  \sigma(s)\,\di \tilde Z(s)  } > K\Big) 
 + \sum_{n= 0}^\infty \P \Big(\sup_{t\in[2^n, 2^{n+1}]} \frac{\int_0^t  \sigma(s)\,\di \tilde Z(s)  }{ 2^{n({\frac{1}{\gamma}+\delta} )} }> K\Big)\\
&\leq \P \Big(\sup_{t\in[0,1]}  {\int_0^t  \sigma(s)\,\di \tilde Z(s)  }  >K\Big) 
+ \sum_{n= 0}^\infty \P \Big(\sup_{t\in[0, 2^{n+1}]} {\int_0^t  \sigma(s)\,\di \tilde Z(s)  } > K 2^{n({\frac{1}{\gamma}+\delta} )}\Big).
    \ea
Let us apply  \eqref{eq:sup_int}, \eqref{eq:1040}, \eqref{eq:1041}
 to the terms in the last line. Note that
all the respective items $I_3$ are zero if $K>K_0=(6C\frac{\gamma}{\gamma-1})^{1/\gamma}$.
Therefore for $C_1=C(1+72\frac{3-\gamma  }{2-\gamma})$ and $K>K_0$ we get
\ba
\P \Big(\sup_{t\geq 0} \frac{\int_0^t  \sigma(s)\,\di \tilde Z(s)  }{ 1+t^{\frac{1}{\gamma}+\delta} } > K\Big) 
\leq \frac{ C_1  }{K^{\gamma}} + \sum_{n\geq 0} \frac{ C_1 2^{n+1} }{K^\gamma{2^{n \gamma(\frac{1}{\gamma}+\delta) }}} 
= \frac{ C_1  }{K^{\gamma}}\Big(1+\frac{2^{2+\gamma\delta}}{2^{1+\gamma\delta}-1}\Big).
\ea
Choosing $K=K(C,\gamma,\delta,\theta)$ large enough we make the last probability less than $\theta$.
\end{proof}

\begin{cor}
\label{cor:upper_limits_noise}
Let $\theta>0$. Let  \emph{\textbf{H}}$_{b}$, \eqref{eq:fin_tail},
\eqref{e:eee1} and \eqref{e:eee2} be satisfied. Let $X^\ve$
be a solution to \eqref{eq:main_small} with any starting point, and let $Y^\ve(t)=X^\e(\e' t )/\e''$, $t\geq 0$,
be the rescaled process.
Then for any $\theta>0$, $T>0$ and $\delta>0$ there exists a generic constant $K=K(\alpha, \delta, \theta,T)$
such that for any
$\ve\in(0,1]$ 
we have
\ba
\label{e:cor}
 \P \Big(\sup_{t\in [0,\frac{T}{\ve'}]}
 \Big|\frac{\int_0^{t}  b_\ve(  Y^\e(s-))\,\di  {Z_\ve(  s)}  }{ 1+t^{\frac{1}{\alpha}+\delta} }\Big|\leq K\Big)
=
\P \Big(\sup_{t\in [0,T]}
 \Big|\frac{\frac{\ve}{ \ve''} \int_0^{t} b(  X^\e(s-))\,\di  {Z (  s)}  }{ 1+(\frac{t}{\ve'})^{\frac{1}{\alpha}+\delta} }\Big| \leq K\Big)
\geq 1-\theta.
\ea
\end{cor}
\begin{proof}
The uniform estimate \eqref{eq:Levy_lim1} from Lemma \ref{l:lm} implies that for any $\gamma\in(1,\alpha)$
there is a constant $C>0$ such that
 the inequalities
\ba
 \int_{|z|>x} \nu_\e(\di z) \leq \frac{C}{x^\gamma},\ x\geq 1\quad \text{and}\quad
 \int_{|z|\leq 1} z^2 \nu_\e(\di z)\leq C,
\ea
hold uniformly over $\e\in(0,1]$.

The only difference between the statement of Lemma
\ref{lem:sup_Levy}
and this corollary is that the processes $\{Z_\e\}$ and the process $Z$ respectively are not necessarily centered and 
that the supremum is taken over a finite $\e$-dependent
interval. Hence we have to estimate the impact of the deterministic drift.
It is more convenient to treat the deterministic linear mean value component $\mu t$ of $Z$, 
$\mu\in\bR$. Indeed, for $\delta>0$ due to \eqref{eq:727} there is a constant $C_1=C_1(\alpha,\delta)$ such that 
$ \e/ \e'' \leq C_1 \cdot (\e')^{-\frac{1}{\alpha} -\delta}$ for $\e\in(0,1]$.
Therefore for some constant $C_2>0$ we have
\ba
\sup_{t\in [0,T]} \Big|\frac{\frac{\ve}{ \ve''}\mu \int_0^t b(  X^\e(s-))\,\di s  }{ 1+(\frac{t}{\ve'})^{\frac{1}{\alpha}+\delta} }\Big|
&\leq \sup_{t\in[0,T]}\frac{t \cdot \frac{\ve }{ \ve''}\cdot  |\mu|\cdot \sup_y|b(y)| }{1+(\frac{t}{\ve'})^{\frac{1}{\alpha}+\delta} }
\leq C_1 \cdot |\mu|\cdot \sup_y|b(y)|\cdot  \sup_{t\in[0,T]}\frac{ t(\ve')^{-\frac{1}{\alpha} -\delta}   }{1+(\frac{t}{\ve'})^{\frac{1}{\alpha}+\delta} }\\
\\
&
=   C_1 \cdot |\mu|\cdot \sup_y|b(y)|\cdot  
\sup_{t\in[0,T]}\frac{ t^{1-\frac{1}{\alpha}-\delta} (\frac{t}{\ve'})^{ \frac{1}{\alpha} +\delta}   }{1+(\frac{t}{\ve'})^{\frac{1}{\alpha}+\delta} }\\
&\leq C_1 \cdot |\mu|\cdot \sup_y|b(y)|\cdot  
 T^{1-\frac{1}{\alpha}-\delta}  \cdot \sup_{s\geq 0}\frac{ s^{\frac{1}{\alpha}+\delta}  }{1+s^{\frac{1}{\alpha}+\delta} }
=: K_0(\alpha,\delta,T),
\ea
that gives us the lower bound for $K$ in \eqref{e:cor}.
\end{proof}

\section{Exit of $X^\e$ from the time-space box $[0, T_0\e']\times [-R\e'',R\e'']$\label{s:exit}}  

In the following Lemma we estimate the exit time of $X^\e$ from a small neighborhood
of $0$. Here we essentially use the representation of $X^\e$ in terms of $Y^\e$ and establish the proper 
relations between its small time and small space behaviour.
For $R>0$ and a stochastic process $X$ we denote the first exit times
\ba
\tau^{X}_R=\inf\{t\geq 0\colon X (t)> R\},\quad \tau^{X}_{-R}=\inf\{t\geq 0\colon X(t)< -R\}.
\ea
 
\begin{lem}
\label{lem:exit-time-is-small}
For any $\theta>0$ and any $R>0$ there is $T_0=T_0(R)>0$ such that
\ba\label{eq:1109}
\liminf_{\e\to 0}\P\Big( \tau^{X^\e}_{R\e''}\wedge \tau^{X^\e}_{-R\e''}\leq T_0\e' 
\Big)\geq 1-\theta.
\ea
\end{lem}

 \begin{proof}
 Recall that $\e'=\e'(\e)$ and $\e''=\e''(\e)$ are chosen according to Lemma \ref{l:eee}. 
 Note that due to rescaling \eqref{eq:Y_rescaling}
\ba
\label{eq:tau-tau}
\tau_{\pm R \ve''}^{X^\ve}= \ve' \tau_{\pm R }^{Y^\ve}, \qquad X^\ve(\tau_{\pm R \ve''}^{X^\ve})= \ve'' Y^\ve(\ve' \tau_{\pm R }^{Y^\ve}).
\ea
Let $R>0$ and choose $\ve_0\in (0,1]$ be such that
\ba
0<\frac{b(0)}{2}\leq \inf_{|y|\leq R,\, \e\in(0,\e_0]}b(\e'' y) =\inf_{|y|\leq R,\, \e\in(0,\e_0]} b_{\e}(y)\leq 
\sup_{|y|\leq R,\, \e\in(0,\e_0]} b_{\e}(y) =\sup_{|y|\leq R,\, \e\in(0,\e_0]}b(\e'' y) \leq 2b(0).
\ea
Also note that $Z_\e\Rightarrow Z^{(\alpha)}$ by Theorem \ref{t:ZY}
so that $Z_\e$ has unbounded jumps in the limit as $\e\to 0$.
Let $\sigma_\e$ be the first jump time such that 
$|\Delta Z_\e(\sigma_\e)|>6R/b(0)$.
Then $|\Delta Y^\ve(\sigma_\e)|>3R$ and hence $\tau_{R}^{Y^\ve}\leq \sigma_\e$.
Eventually \eqref{e:ZZ} yields
\ba
\lim_{\e\to 0}\E \sigma_\e = \Big(\int_{|z|>6R/b(0)} \nu^{(\alpha)}(\di z)\Big)^{-1}
\ea
and the statement of the Lemma follows from \eqref{eq:tau-tau} and Chebyshev's inequality.
 \end{proof}

\begin{cor}\label{cor:exitX}
For any $\theta>0$ 
there exist $R>0$ large enough and $T_0>0$ 
such that
\ba
&\limsup_{\e\to 0}\Big|\P\Big(\tau_{R\e''}^{X^\e}<\tau_{-R\e''}^{X^\e}\leq T_0\e'\Big)-\bar p_+\Big|\leq \theta,\\
&\limsup_{\e\to 0}\Big|\P\Big(\tau_{-R\ve''}^{X^\e}<\tau_{R\ve''}^{X^\e}\leq T_0\e' \Big)-\bar p_-\Big|\leq \theta.
\ea

\end{cor}
\begin{proof}
The result follows from \eqref{eq:1109}, \eqref{eq:tau-tau}, and  Theorems \ref{thm:PP} and \ref{t:ZY}.
\end{proof}

\section{Behaviour of $X^\e$ upon exit from the time-space box $[0,T_0\e']\times [-R\e'',R\e'']$\label{s:exit2}. Proof of the 
main result}  

For definiteness, let us consider only dynamics on the positive half line.

\begin{lem}
1. For each $\gamma\in (0,\beta)$ there is $K_\gamma>0$ such that for all $x\geq 1$ and $\e\in(0,1]$
\ba
\label{e:Kgamma}
a_\e(x)\geq K_\gamma x^{\beta-\gamma}.
\ea
2. For any $\kappa\in(0,1)$ there exists $\mu\in(0,1)$ such that
\ba
\label{eq:a-a-estim}
\inf_{\frac{x}{y}\in[1-\mu, 1+\mu]}\frac{a_\ve(x)}{a_\ve(y)} 
> 1-\kappa .
\ea
\end{lem}
\begin{proof}
1. Recall that according to Assumption \textbf{H}$_a$ and Remark \ref{r:truncate}, $a(x)=x^\beta L_+(x\wedge 1)$, $x>0$, and $a(0)=0$.
Hence
\ba
\label{e:qw}
a_\e(x)=\frac{a(\e'' x)}{\e''/\e'}&=\e'\cdot (\e'')^{\beta-1}\cdot x^\beta L_+((\e'' x)\wedge 1) \\
&=\e'\cdot (\e'')^{\beta-1}\cdot l\Big(\frac{1}{\e''}\Big)\cdot x^\beta \cdot \frac{L_+((\e'' x)\wedge 1)}{A_+ l\big(\frac{1}{x\e''}\vee 1\big)}
\cdot\frac{A_+ l\big(\frac{1}{x\e''}\vee 1\big)}{l\big(\frac{1}{\e''}\big) }.
\ea
The equivalence \eqref{e:eee1} guarantees that $\e'\cdot (\e'')^{\beta-1}\cdot l(\frac{1}{\e''})\geq C_1>0$ for some $C_1>0$ and $\e\in(0,1]$. 

Let $\gamma\in (0,\beta)$, $x\geq 1$ and $\e\in(0,1]$. We consider two cases.

\noindent
a) For $x\e''< 1$,  with the help of Potter's theorem \cite[Theorem 1.5.6 (ii)]{BinghamGT-87} applied to the function $l$ we get
\ba
a_\e(x)\geq C_1 \cdot x^\beta \cdot \inf_{0<y<1}\frac{L_+(y)}{A_+ l\big(\frac{1}{y}\big)}
\cdot A_+\cdot \frac{l\big(\frac{1}{x\e''}\big)}{l\big(\frac{1}{\e''}\big) }
\geq C_2 \cdot x^{\beta-\gamma}
\ea
for some $C_2=C_2(\gamma)>0$.

\noindent
b) For $x\e''\geq 1$ applying Potter's theorem again we get 
\ba
a_\e(x)\geq C_1 \cdot x^\beta \cdot \frac{L_+(1)}{ l(1)}
\cdot\frac{ l(1)}{l\big(\frac{1}{\e''}\big) }\geq C_2\cdot  x^\beta \cdot (\e'')^\gamma\geq C_3\cdot  x^{\beta-\gamma}
\ea
for some $C_3=C_3(\gamma)$, and \eqref{e:Kgamma} follows with $K_\gamma=C_2\wedge C_3$.

\noindent 
2. 
To prove \eqref{eq:a-a-estim} we note that
\ba
\inf_{\frac{x}{y}\in[1-\mu, 1+\mu]}\frac{a_\ve(x)}{a_\ve(y)} 
= \inf_{\frac{x}{y}\in[1-\mu, 1+\mu]}\frac{a(x)}{a(y)} 
=(1-\mu)^\beta \cdot \inf_{\frac{x}{y}\in[1-\mu, 1+\mu]}\frac{L_+(x\wedge 1)}{L_+(y\wedge 1)} 
=C(\mu)\cdot (1-\mu)^\beta, 
\ea
where $0<C(\mu)\to 1$ as $\mu\to 0$ by continuity of $L_+$ and Potter's bounds. Hence for any $\kappa\in (0,1)$, the estimate 
\eqref{eq:a-a-estim} holds for $\mu$ small enough. 
\end{proof}

\begin{lem}
Let $\gamma\in(0,\beta)$.
Then for any $y\geq 1$, $\kappa\in(0,1)$ and any $\e\in(0,1]$ the solution of the  ODE
\ba
\zeta^\ve_{\kappa}(t;y) = y + (1-\kappa) \int_0^t a_\ve(\zeta_{\kappa}^\e(s;y))\, \di s
\ea
satisfies
\ba
\label{eq:1265}
\zeta^\ve_{\kappa}(t;y)\geq y+K t^{\frac{1}{1-\beta+ \gamma}},\ t\geq 0,
\ea
with a constant $K=K(\beta,\gamma,\kappa)>0$.
\end{lem}
\begin{proof}
Let $\gamma\in (0,\beta)$ be fixed.
For $y\geq 1$ we use \eqref{e:Kgamma} and compare $\zeta^\ve_{\kappa}(\cdot; y)$ with the solution of the auxiliary ODE
\ba
z_{\kappa}(t;y) = y + (1-\kappa)K_\gamma \int_0^t (z_{\kappa }(s;y))^{\beta-\gamma}\, \di s,\ t\geq 0.
\ea
This solution has the explicit form 
\ba
z_{\kappa}(t;y)=\Big(y^{1-\beta+\gamma}+(1-\kappa)(1-\beta+\gamma)K_\gamma t\Big)^{\frac{1}{1-\beta+ \gamma}}.
\ea
Hence the application of an elementary inequality $(a+b)^p\geq a^p+b^p$, $a,b\geq 0$, $p\geq 1$, yields \eqref{eq:1265} with some $K>0$.
\end{proof}

We need the following comparison theorem for solutions of integral equations.
\begin{lem}
\label{lem:compar}
Let for $T>0$ and $i=1,2$, the functions $u_i$ be solutions (not necessarily unique) to the equations
\ba
u_i(t) = u_i(0) +\int_0^t U_i(s, u_i(s))\, \di s, \ t\in[0,T].
\ea
Assume that $u_1(0)\geq u_2(0)$,  $U_1(t, u_2(t)) >  U_2(t, u_2(t))$, $t\in[0,T]$, and 
functions $t\mapsto U_i(t, u_i(t))$ are right-continuous.
Then $u_1(t)\geq u_2(t)$, $t\in[0,T]$.
\end{lem}
\begin{proof}
The proof of this Lemma is quite standard. Assume that there is 
\ba
\tau=\inf\{t>0\colon u_1(t)<u_2(t)\}\in[0,T].
\ea
Then by continuity $u_1(\tau)=u_2(\tau)$ and we necessarily have the inequality $D^+ u_1(\tau)\leq D^+u_2(\tau)$ for the right Dini derivatives of the solutions.
However since $t\mapsto U_i(t,u_i(t))$ is right-continuous, by assumption
\ba
D^+ u_1(\tau)=U_1(\tau,u_1(\tau))=U_1(\tau,u_2(\tau))> U_2(\tau,u_2(\tau))=D^+u_2(\tau),
\ea
and we obtain a contradiction. 
\end{proof}

In the next Lemma we determine a lower bound for the process $Y^\e$ starting sufficiently far from zero.

\begin{lem}  
\label{lem:972}
For any  $\theta>0$, 
$\kappa\in(0,1)$ and 
$T>0$ there are $\mu=\mu(\kappa)\in(0,1)$ and 
$R=R(T,\kappa,\theta)\geq 1$
 such that for
any $\cF_0$-measurable initial condition $Y^\e(0)> R$ a.s.
and all $\e\in(0,1]$
\ba
\P\Big(Y^\e(t) \geq (1-\mu)\zeta_{\kappa}^\e (t; Y^\e(0))
,\ t\in[0,T/\e'] \Big)\geq 1-\theta.
\ea
A similar estimate from above also holds for $Y^\e(0)< -R$ a.s.
\end{lem}
\begin{proof}
For 
$\e\in(0,1]$ let
\ba
g^\ve(t):=\int_0^t  b_\e(  Y^\e(s-))\,\di Z_\e(s), \quad
\tilde Y^\ve(t):= Y^\ve(t) - g^\ve(t).
\ea
Then 
$\tilde Y^\ve(t)$ satisfies the integral equation
\ba
\tilde Y^\ve(t) = Y^\e(0) +\int_0^t a_\ve(\tilde Y^\ve(s)+ g^\ve(s))\, \di s.
\ea
Choose $\gamma\in(0,\beta)$ small enough such that $\frac{1}{1-\beta+\gamma} >\frac{1}{\alpha}$. For $\theta>0$ fixed, we apply Corollary \ref{cor:upper_limits_noise}
and find a constant $K_1=K_1(T,\beta,\gamma,\theta)>0$
 such that for all $\e\in(0,1]$ and any initial starting point $Y^\e(0)\in\bR$
\ba
\label{eq:990}
\P \Big(\sup_{t\in[0,\frac{T}{\e'}]} \Big|\frac{g^\e(t)}{ 1+t^{\frac{1}{1-\beta+\gamma} } }\Big| \leq  K_1 \Big)
\geq 1-\theta.
\ea
Consequently, for any $\kappa\in (0,1)$ and any $y\geq 1$ with the help of \eqref{eq:1265} we get
\ba
\P \Big(\sup_{t\in[0,\frac{T}{\e'}]} \Big|\frac{g^\e(t)}{ \zeta^\e_{\kappa}(t;y) } \Big| 
\leq  \frac{K_1(1+t^{\frac{1}{1-\beta+\gamma} })}{y +K t^{\frac{1}{1-\beta+\gamma} }} \Big)
\geq 1-\theta. 
\ea
Let $\mu=\mu(\kappa)\in(0,1)$ be such that \eqref{eq:a-a-estim} holds. 
For this $\mu$ 
choose $R\geq 1$ such that
 $\sup_{t\geq 0}\frac{K_1(1+t^{\frac{1}{1-\beta+\gamma} })}{R+K t^{\frac{1}{1-\beta+\gamma} }}\leq \mu$.
 Then 
\ba
\P \Big(\sup_{t\in[0,\frac{T}{\e'}]} \Big|\frac{g^\e(t)}{ \zeta^\e_{\kappa}(t;R\vee Y^\e(0)) }\Big| \leq \mu \Big)\geq 1-\theta.
\ea
In other words, for $Y^\e(0)\geq R$ a.s.\ we have
\ba
\P \Big( a_\e(\zeta^\e_{\kappa}(t;Y^\e(0)) + g^\e(t))> (1-\kappa) a_\e(\zeta^\e_{\kappa}(t; Y^\e(0))),\ t\in  [0,T/\e'] \Big)\geq 1-\theta.
\ea
Therefore the comparison Lemma \ref{lem:compar} applied to $u_1=\tilde Y^\e$ and $u_2=\zeta^\e_{\kappa }(\cdot; Y^\e(0) )$ yields 
\ba
\P \Big( \tilde Y^\e (t)\geq \zeta^\e_{\kappa}(t; Y^\e(0)),\ t\in  [0,T/\e']  \Big)\geq 1-\theta
\ea
and hence
\ba
\P \Big(Y^\e (t)\geq (1-\mu) \zeta^\e_{\kappa}(t; Y^\e(0))
,\ t\in  [0, T/\e'] \Big)\geq 1-\theta.
\ea
\end{proof}

\noindent
\textit{Proof of Theorem \ref{thm:main}}. 

Notice that for each $\kappa\in(0,1)$, $\e\in(0,1]$ and $y> 0$ the function 
$\hat \zeta^\e_{\kappa}(t;y):=\ve''\zeta^\ve_{\kappa}(t/\e';y)$, $t\geq 0$, satisfies the equation
\ba 
\hat \zeta^\e_{\kappa}(t;y)= \e'' y + (1-\kappa)\int_0^t a(\hat \zeta^\ve_{\kappa}(s;y))\, \di s.
\ea
Hence according to \eqref{e:ode}, \eqref{eq:solutionODE} and \eqref{e:eq_pm}
\ba
\label{eq:1331}
\hat \zeta^\e_{\kappa}(t;y) = X^0_{\e'' y}((1-\kappa)t) \geq x^+((1-\kappa)t),\quad t\geq 0.
\ea
Let $\mu=\mu(\kappa)\in(0,1)$ be chosen to satisfy \eqref{eq:a-a-estim}.

Since the L\'evy process  $Z_\ve$ is strong Markov,  analogously to
Lemma \ref{lem:sup_Levy} and Corollary \ref{cor:upper_limits_noise} we have the following.
For any $T$, $\delta$, $\theta>0$ there exists a generic constant $K=K(T, \alpha, \delta, \theta)$
such that for any
$\ve\in(0,1]$ the estimate
\ba
 \P \Big(
\sup_{t\in [0,\frac{T}{\ve'}]}
\Big|
\frac{\int_{\tau}^{\tau+t}  b_\ve(  Y^\e(s-))\,\di  
{Z_\ve(  s)}  }{ 1+t^{\frac{1}{\alpha}+\delta} }\Big|\ \leq K\Big)
   \geq 1-\theta.
\ea
holds for any stopping time $\tau$.
 It follows from 
Corollary \ref{cor:upper_limits_noise},
Lemma \ref{lem:exit-time-is-small},
Corollary  \ref{cor:exitX},
Lemma \ref{lem:972}, and
\eqref{eq:1331}
that for any $\theta>0$ and $T>0$
there are  $R>0$ and $T_0>0$ large enough such that
\ba
& \liminf_{\e\to 0}\P\Big( \tau^{X^\e}_{R\e''}< \tau^{X^\e}_{-R\e''}\leq T_0\e',\ \ 
 X^\e(\tau^{X^\e}_{R\e''} + t)\geq  (1-\mu)x^+((1-\kappa)t),\ t\in[0,T]
 \Big)\geq \bar p_+-\theta,\\
& \liminf_{\ve\to 0}\P\Big(\tau^{X^\e}_{-R\e''}< \tau^{X^\e}_{R\e''}\leq T_0\e',\ 
 \ 
X^\e(\tau^{X^\e}_{-R\e''} + t)\leq  (1-\mu)x^-((1-\kappa)t),\ t\in[0,T] \Big)\geq \bar p_--\theta.
 \ea
In the last formula, Lemma  \ref{lem:972} is applied to the process $Y^\e(t+ \tau^{Y^\e}_{-R}\wedge \tau^{Y^\e}_{R} )$, $t\geq 0$,
whose initial value belongs to the set $[-R,R]^c$, see \eqref{eq:tau-tau}. Corollary \ref{cor:upper_limits_noise} holds true since 
$\tau^{Y^\e}_{\pm R}$ are stopping times.

Since $\bar p_-+\bar p_+=1$ and any limit law of $\{X^\e\}$ is supported by the solutions $x^\pm$ (see Lemma \ref{l:X}) we get
that for each $\delta>0$
\ba
& \limsup_{\e\to 0}\Big|\P\Big( \sup_{t\in[0,T_0\e'+T]}|X^\e(t)-x^+(t)|\leq \delta \Big)- p_+\Big|\leq \theta,\\
& \limsup_{\e\to 0}\Big|\P\Big( \sup_{t\in[0,T_0\e'+T]}|X^\e(t)-x^-(t)|\leq \delta \Big)- p_-\Big|\leq \theta,
\ea
and the proof is finished. 
\hfill$\Box$

%

\begin{thebibliography}{40}
\providecommand{\natexlab}[1]{#1}
\providecommand{\url}[1]{\texttt{#1}}
\expandafter\ifx\csname urlstyle\endcsname\relax
  \providecommand{\doi}[1]{doi: #1}\else
  \providecommand{\doi}{doi: \begingroup \urlstyle{rm}\Url}\fi

\bibitem[Alabert and Le\'on(2017)]{AlabLeon-17}
A.~Alabert and J.~A. Le\'on.
\newblock {On uniqueness for some non-Lipschitz SDE}.
\newblock \emph{Journal of Differential Equations}, 262\penalty0 (12):\penalty0
  6047--6067, 2017.

\bibitem[Amine et~al.(2017)Amine, Ba{\~n}os, and Proske]{banos2017c}
O.~Amine, D.~R. Ba{\~n}os, and F.~N. Proske.
\newblock {$C^\infty$}-regularization by noise of singular {ODE}'s.
\newblock \emph{arXiv:1710.05760}, 2017.

\bibitem[Bafico and Baldi(1982)]{BaficoB-82}
R.~Bafico and P.~Baldi.
\newblock {Small random perturbations of Peano phenomena}.
\newblock \emph{Stochastics}, 6\penalty0 (3--4):\penalty0 279--292, 1982.

\bibitem[Ba{\~n}os et~al.(2019{\natexlab{a}})Ba{\~n}os, Bauer, Meyer-Brandis,
  and Proske]{banos2019restoration}
D.~Ba{\~n}os, M.~Bauer, T.~Meyer-Brandis, and F.~Proske.
\newblock Restoration of well-posedness of infinite-dimensional singular
  {ODE}'s via noise.
\newblock \emph{arXiv:1903.05863}, 2019{\natexlab{a}}.

\bibitem[Ba{\~n}os et~al.(2019{\natexlab{b}})Ba{\~n}os, Nilssen, and
  Proske]{banos2015strong}
D.~Ba{\~n}os, T.~Nilssen, and F.~Proske.
\newblock Strong existence and higher order {F}r{\'e}chet differentiability of
  stochastic flows of fractional {B}rownian motion driven {SDE}s with singular
  drift.
\newblock \emph{Journal of Dynamics and Differential Equations},
  2019{\natexlab{b}}.

\bibitem[Ba{\~n}os et~al.(2018)Ba{\~n}os, Duedahl, Meyer-Brandis, and
  Proske]{banos2018construction}
D.~R. Ba{\~n}os, S.~Duedahl, T.~Meyer-Brandis, and F.~Proske.
\newblock Construction of {M}alliavin differentiable strong solutions of {SDE}s
  under an integrability condition on the drift without the
  {Y}amada--{W}atanabe principle.
\newblock \emph{Annales de l'Institut Henri Poincar{\'e}, Probabilit{\'e}s et
  Statistiques}, 54\penalty0 (3):\penalty0 1464--1491, 2018.

\bibitem[Barrimi and Ouknine(2016)]{barrimi2017approximation}
O.~E. Barrimi and Y.~Ouknine.
\newblock Approximation of solutions of {SDE}s driven by a fractional
  {B}rownian motion, under pathwise uniqueness.
\newblock \emph{Modern Stochastics: Theory and Applications}, 3\penalty0
  (4):\penalty0 303--313, 2016.

\bibitem[Bingham et~al.(1987)Bingham, Goldie, and Teugels]{BinghamGT-87}
N.~H. Bingham, C.~M. Goldie, and J.~L. Teugels.
\newblock \emph{{Regular Variation}}, volume~27 of \emph{{Encyclopedia of
  Mathematics and its Applications}}.
\newblock Cambridge University Press, Cambridge, 1987.

\bibitem[Buckdahn et~al.(2009)Buckdahn, Ouknine, and
  Quincampoix]{buckdahn2009limiting}
R.~Buckdahn, Y.~Ouknine, and M.~Quincampoix.
\newblock On limiting values of stochastic differential equations with small
  noise intensity tending to zero.
\newblock \emph{Bulletin des Sciences Mathematiques}, 133\penalty0
  (3):\penalty0 229--237, 2009.

\bibitem[Buldygin et~al.(2018)Buldygin, Indlekofer, Klesov, and
  Steinebach]{buldygin2018pseudo}
V.~V. Buldygin, K.-H. Indlekofer, O.~I. Klesov, and J.~G. Steinebach.
\newblock \emph{Pseudo-Regularly Varying Functions and Generalized Renewal
  Processes}, volume~91 of \emph{Probability Theory and Stochastic Modelling}.
\newblock Springer, Cham, 2018.

\bibitem[Catellier and Gubinelli(2016)]{catellier2016averaging}
R.~Catellier and M.~Gubinelli.
\newblock Averaging along irregular curves and regularisation of {ODE}s.
\newblock \emph{Stochastic Processes and their Applications}, 126\penalty0
  (8):\penalty0 2323--2366, 2016.

\bibitem[Chen and Wang(2016)]{chen2016uniqueness}
Z.-Q. Chen and L.~Wang.
\newblock Uniqueness of stable processes with drift.
\newblock \emph{{Proceedings of the American Mathematical Society}},
  144\penalty0 (6):\penalty0 2661--2675, 2016.

\bibitem[Davie(2007)]{davie2007uniqueness}
A.~M. Davie.
\newblock Uniqueness of solutions of stochastic differential equations.
\newblock \emph{International Mathematics Research Notices}, 2007.
\newblock Article ID rnm124.

\bibitem[Davie(2011)]{davie2011}
A.~M. Davie.
\newblock Individual path uniqueness of solutions of stochastic differential
  equations.
\newblock In D.~Crisan, editor, \emph{{Stochastic Analysis 2010}}, pages
  213--225. Springer, Berlin, 2011.

\bibitem[Delarue and Flandoli(2014)]{delarue2014transition}
F.~Delarue and F.~Flandoli.
\newblock {The transition point in the zero noise limit for a 1D Peano
  example}.
\newblock \emph{Discrete \& Continuous Dynamical Systems --- A}, 34\penalty0
  (10):\penalty0 4071--4083, 2014.

\bibitem[Delarue and Maurelli(2019)]{delarue2019zero}
F.~Delarue and M.~Maurelli.
\newblock {Zero noise limit for multidimensional SDEs driven by a pointy
  gradient}.
\newblock \emph{arXiv preprint arXiv:1909.08702}, 2019.

\bibitem[Flandoli(2011{\natexlab{a}})]{Flandoli-2011}
F.~Flandoli.
\newblock \emph{{Random Perturbation of PDEs and Fluid Dynamic Models.
  {\'E}cole d'{\'E}t{\'e} de Probabilit{\'e}s de Saint--Flour XL -- 2010}},
  volume 2015 of \emph{{Lecture Notes in Mathematics}}.
\newblock Springer, Berlin, 2011{\natexlab{a}}.

\bibitem[Flandoli(2011{\natexlab{b}})]{flandoli2011regularizing}
F.~Flandoli.
\newblock Regularizing properties of {B}rownian paths and a result of {D}avie.
\newblock \emph{Stochastics and Dynamics}, 11\penalty0 (02n03):\penalty0
  323--331, 2011{\natexlab{b}}.

\bibitem[Galeati and Gubinelli(2020)]{galeatigubinelli20}
L.~Galeati and M.~Gubinelli.
\newblock Noiseless regularisation by noise.
\newblock \emph{arXiv:2003.14264}, 2020.

\bibitem[Gikhman and Skorokhod(1982)]{GihSko-82}
I.~I. Gikhman and A.~V. Skorokhod.
\newblock \emph{{Stokhasticheskie differentsial$'$nye uravneniya i ikh
  prilozheniya} (Russian) [Stochastic differential equations and their
  applications]}.
\newblock Naukova Dumka, Kiev, 1982.

\bibitem[Gradinaru et~al.(2001)Gradinaru, Herrmann, and
  Roynette]{gradinaru2001singular}
M.~Gradinaru, S.~Herrmann, and B.~Roynette.
\newblock A singular large deviations phenomenon.
\newblock \emph{Annales de l'Institut Henri Poincar\'e. Probabilit\'es et
  Statistiques}, 37\penalty0 (5):\penalty0 555--580, 2001.

\bibitem[Harang and Perkowski(2020)]{harangperkowski20}
F.~A. Harang and N.~Perkowski.
\newblock {$C^\infty$}-regularization of {ODE}s perturbed by noise.
\newblock \emph{arXiv:2003.05816}, 2020.

\bibitem[Hartman(1964)]{Hartman-64}
P.~Hartman.
\newblock \emph{Ordinary Differential Equations}.
\newblock John Wiley \& Sons, New York, 1964.

\bibitem[Jacod and Shiryaev(2003)]{JacodS-03}
J.~Jacod and A.~N. Shiryaev.
\newblock \emph{{Limit Theorems for Stochastic Processes}}, volume 288 of
  \emph{{Grundlehren der Mathematischen Wissenschaften}}.
\newblock Springer, Berlin, second edition, 2003.

\bibitem[Krykun and Makhno(2013)]{krykun2013peano}
I.~G. Krykun and S.~Ya. Makhno.
\newblock The {P}eano phenomenon for {I}t{\^o} equations.
\newblock \emph{Journal of Mathematical Sciences}, 192\penalty0 (4):\penalty0
  441--458, 2013.

\bibitem[Krylov and R\"ockner(2005)]{krylov2005strong}
N.~V. Krylov and M.~R\"ockner.
\newblock Strong solutions of stochastic equations with singular time dependent
  drift.
\newblock \emph{Probability Theory and Related Fields}, 131\penalty0
  (2):\penalty0 154--196, 2005.

\bibitem[Kulik(2019)]{kulik2019weak}
A.~M. Kulik.
\newblock {On weak uniqueness and distributional properties of a solution to an
  SDE with $\alpha$-stable noise}.
\newblock \emph{Stochastic Processes and their Applications}, 129\penalty0
  (2):\penalty0 473--506, 2019.

\bibitem[Pilipenko and Proske(2018{\natexlab{a}})]{pilipenko2018perturbations}
A.~Pilipenko and F.~N. Proske.
\newblock {On perturbations of an ODE with non-Lipschitz coefficients by a
  small self-similar noise}.
\newblock \emph{Statistics \& Probability Letters}, 132:\penalty0 62--73,
  2018{\natexlab{a}}.

\bibitem[Pilipenko and Proske(2018{\natexlab{b}})]{pilipenko2018selection}
A.~Pilipenko and F.~N. Proske.
\newblock On a selection problem for small noise perturbation in the
  multidimensional case.
\newblock \emph{Stochastics and Dynamics}, 18\penalty0 (6):\penalty0 1850045,
  2018{\natexlab{b}}.

\bibitem[Priola(2012)]{Priola-12}
E.~Priola.
\newblock {Pathwise uniqueness for singular SDEs driven by stable processes}.
\newblock \emph{Osaka Journal of Mathematics}, 49\penalty0 (2):\penalty0
  421--447, 2012.

\bibitem[Priola(2018)]{priola2018davie}
E.~Priola.
\newblock Davie's type uniqueness for a class of {SDE}s with jumps.
\newblock \emph{Annales de l'Institut Henri Poincar{\'e}. Probabilit{\'e}s et
  Statistiques}, 54\penalty0 (2):\penalty0 694--725, 2018.

\bibitem[Pruitt(1981)]{pruitt1981growth}
W.~E. Pruitt.
\newblock The growth of random walks and {L\'e}vy processes.
\newblock \emph{The Annals of Probability}, 9\penalty0 (6):\penalty0 948--956,
  1981.

\bibitem[Shaposhnikov(2016)]{shaposhnikov2016some}
A.~V. Shaposhnikov.
\newblock Some remarks on {D}avie's uniqueness theorem.
\newblock \emph{Proceedings of the Edinburgh Mathematical Society}, 59\penalty0
  (4):\penalty0 1019--1035, 2016.

\bibitem[Situ(2005)]{Situ05}
R.~Situ.
\newblock \emph{{Theory of Stochastic Differential Equations with Jumps and
  Applications}}.
\newblock Springer, New York, NY, 2005.

\bibitem[Strook and Varadhan(1979)]{StrVar}
D.~Strook and S.~R.~S. Varadhan.
\newblock \emph{{Multidimensional Diffusion Processes}}, volume 233 of
  \emph{{Grundlehren der Mathematischen Wissenschaften}}.
\newblock Springer, Berlin, 1979.

\bibitem[Tanaka et~al.(1974)Tanaka, Tsuchiya, and Watanabe]{TanTsuWat74}
H.~Tanaka, M.~Tsuchiya, and S.~Watanabe.
\newblock {Perturbation of drift-type for {L}{\'e}vy processes}.
\newblock \emph{Journal of Mathematics of Kyoto University}, 14\penalty0
  (1):\penalty0 73--92, 1974.

\bibitem[Trevisan(2013)]{trevisan2013zero}
D.~Trevisan.
\newblock Zero noise limits using local times.
\newblock \emph{Electronic Communications in Probability}, 18\penalty0
  (31):\penalty0 1--7, 2013.

\bibitem[Veretennikov(1981)]{veretennikov1981strong}
A.~Yu. Veretennikov.
\newblock On strong solutions and explicit formulas for solutions of stochastic
  integral equations.
\newblock \emph{Mathematics of the {USSR}--{S}bornik}, 39\penalty0
  (3):\penalty0 387--403, 1981.

\bibitem[Veretennikov(1983)]{veretennikov1983approximation}
A.~Yu. Veretennikov.
\newblock Approximation of ordinary differential equations by stochastic
  differential equations.
\newblock \emph{Mathematical Notes of the Academy of Sciences of the USSR},
  33\penalty0 (6):\penalty0 476--477, 1983.

\bibitem[Zvonkin(1974)]{zvonkin74}
A.~K. Zvonkin.
\newblock {A transformation of the phase space of a diffusion process that
  removes the drift}.
\newblock \emph{Mathematics of the USSR--Sbornik}, 22\penalty0 (1):\penalty0
  129, 1974.

\end{thebibliography}
%

\end{document}